\documentclass[11pt,a4paper]{article}

\usepackage[T1]{fontenc}
\usepackage[utf8]{inputenc}
\usepackage{lmodern}
\usepackage{amsmath,amssymb,amsthm,mathtools}
\usepackage{microtype}
\usepackage{xcolor}
\usepackage{xurl}
\usepackage[a4paper,margin=30mm]{geometry}
\usepackage[
  colorlinks=true,
  linkcolor=blue!45!black,
  citecolor=blue!45!black,
  urlcolor=blue!55!black,
]{hyperref}

\allowdisplaybreaks[2]

\newtheorem{theorem}{Theorem}[section]
\newtheorem{proposition}[theorem]{Proposition}
\newtheorem{lemma}[theorem]{Lemma}
\newtheorem{corollary}[theorem]{Corollary}
\theoremstyle{definition}
\newtheorem{definition}[theorem]{Definition}
\theoremstyle{remark}
\newtheorem{remark}[theorem]{Remark}

\newcommand{\OO}{\mathcal O}
\newcommand{\Apol}{\mathcal A}
\newcommand{\Gpol}{\mathcal G}
\newcommand{\Fpol}{\mathcal F}

\title{Two-Line Elimination and the Weierstrass Group Law\\
in Arbitrary Characteristic}
\author{Hongfeng Wu\\
College of Science, North China University of Technology\\
\texttt{whfmat@gmail.com}}
\date{}

\begin{document}

\maketitle

\begin{abstract}
Let $K$ be a field and let $E/K$ be a nonsingular projective curve in general Weierstrass form.
We prove associativity of its chord--tangent operation by combining polynomial factorization, formal differentiation, and quadratic interpolation.
The main identity eliminates the two nonvertical lines $Y=\ell(X)$ and $Y=m(X)$ occurring in the successive additions $S=P\oplus Q$ and $T=S\oplus R$.
There are $c\in K$ and a monic quadratic polynomial $h\in K[X]$ such that
\[
 \ell+m+\Apol=c(X-x_S),
 \qquad
 \Gpol+\ell m=(X-x_S)h,
\]
and
\[
 (Y-\ell)(Y-m)
 =(X-x_S)(h-cY)
 +Y^2+\Apol(X)Y-\Gpol(X).
\]
When $c\ne0$, the graph of $q=h/c$ contains $P,Q,R$, and $-T$.
The two bracketings determine the same quadratic by ordinary or Hermite interpolation and therefore yield the same final point.
The proof applies in every characteristic, since tangency is expressed through the formal partial derivatives of the Weierstrass equation.
For the short Weierstrass form, a separate concise proof identifies $q$ with the interpolating parabola in Zwegers's construction.
\end{abstract}



\section{Introduction}

The chord--tangent operation on a nonsingular plane cubic is geometrically natural, but its associativity is not immediate from the construction.
Conceptual proofs use the Picard group or related algebro-geometric machinery \cite[Chapter~III]{Silverman2009}; other proofs use complex uniformization, the Cayley--Bacharach theorem, or direct coordinate calculation.
Elementary treatments based on explicit formulas include Friedl's computational argument \cite{Friedl2017}, Nuida's linear-algebraic proof for a general Weierstrass equation \cite{Nuida2021}, and the formally verified arbitrary-characteristic proof of Angdinata and Xu \cite{AngdinataXu2023}.
In 2024, Zwegers gave a particularly short argument for a short Weierstrass equation by deriving a symmetric formula for a sum of three points and interpreting the formula through an interpolating parabola \cite{Zwegers2024}.

The central construction is a low-degree polynomial identity that extends the parabola method directly to
\begin{equation}\label{eq:intro-general}
 Y^2+a_1XY+a_3Y
 =X^3+a_2X^2+a_4X+a_6
\end{equation}
over an arbitrary field.
Its ingredients are factorization in $K[X]$, formal differentiation, and uniqueness of interpolation by polynomials of degree at most two.
All multiplicities that occur are double roots of univariate polynomials and are recorded by first derivatives.

The first line in the computation of $(P\oplus Q)\oplus R$ passes through $P,Q,-S$, where $S=P\oplus Q$, whereas the second passes through $S,R,-T$, where $T=(P\oplus Q)\oplus R$.
Their values above $x_S$ are conjugate roots of the Weierstrass equation.
Their sum can therefore be divided by $X-x_S$, as can a second expression obtained from their product.
The resulting identity produces a quadratic polynomial whose graph contains $P,Q,R,-T$ with the required tangency when points coincide.
Repeating the construction for $P\oplus(Q\oplus R)$ gives a second quadratic with the same interpolation data.
Degree two leaves no freedom, so the two quadratics and then the two final sums coincide.

Section~\ref{sec:preliminaries} records the elementary chord--tangent facts used without assuming associativity.
Section~\ref{sec:identity} proves the two-line elimination identity.
Section~\ref{sec:short} gives a separate concise proof for the short Weierstrass form.
Section~\ref{sec:general-associativity} proves associativity for the general equation, and Section~\ref{sec:small-characteristic} makes the specializations to characteristics $2$ and $3$ explicit.
Section~\ref{sec:comparison} identifies the short-form quadratic with Zwegers's parabola and compares the proof with other elementary approaches.

\section{Elementary chord--tangent preliminaries}\label{sec:preliminaries}

Let $K$ be a field of arbitrary characteristic, let $a_1,a_2,a_3,a_4,a_6\in K$, and let $E\subset\mathbb P^2_K$ be the nonsingular cubic
\begin{equation}\label{eq:projective-general}
 Y^2Z+a_1XYZ+a_3YZ^2
 =X^3+a_2X^2Z+a_4XZ^2+a_6Z^3.
\end{equation}
Setting $Z=0$ in \eqref{eq:projective-general} gives $X=0$, so the unique point at infinity is
\[
 \OO=[0:1:0].
\]
At the representative $[0:1:0]$, the partial derivative with respect to $Z$ of the left side minus the right side of \eqref{eq:projective-general} is $1$; in particular, $\OO$ is nonsingular.
On the affine chart $Z=1$, equation \eqref{eq:projective-general} becomes \eqref{eq:intro-general}.
Put
\begin{equation}\label{eq:AGF}
 \Apol(X)=a_1X+a_3,
 \qquad
 \Gpol(X)=X^3+a_2X^2+a_4X+a_6,
 \qquad
 \Fpol(X,Y)=Y^2+\Apol(X)Y-\Gpol(X).
\end{equation}
Thus the affine part of $E$ is given by $\Fpol=0$.

For a finite point $P=(x_P,y_P)$, define
\begin{equation}\label{eq:negative}
 -P=\bigl(x_P,-\Apol(x_P)-y_P\bigr),
 \qquad
 -\OO=\OO.
\end{equation}
The ordinates $y_P$ and $-\Apol(x_P)-y_P$ are the two roots, counted with multiplicity, of the monic quadratic $\Fpol(x_P,Y)$.
Thus \eqref{eq:negative} exchanges the two points above a fixed affine $x$-coordinate.
Consequently, two finite points of $E$ have the same $x$-coordinate if and only if they are equal or negatives of one another.
The formal partial derivatives are
\begin{equation}\label{eq:partials}
 \Fpol_X(X,Y)=a_1Y-\Gpol'(X),
 \qquad
 \Fpol_Y(X,Y)=2Y+\Apol(X),
\end{equation}
where
\[
 \Gpol'(X)=3X^2+2a_2X+a_4.
\]
In particular,
\begin{equation}\label{eq:self-inverse}
 P=-P
 \quad\Longleftrightarrow\quad
 \Fpol_Y(P)=2y_P+\Apol(x_P)=0.
\end{equation}
If a finite point is self-inverse, nonsingularity gives $\Fpol_X(P)\ne0$, so its tangent is vertical.

\begin{definition}[Chord--tangent operation]\label{def:addition}
Set $P\oplus\OO=\OO\oplus P=P$ and, for a finite point $U$, set $U\oplus(-U)=\OO$.
Suppose that $U,V$ are finite and $V\ne-U$.
The relevant nonvertical chord or tangent is
\[
 Y=\lambda_{U,V}(X)
 =y_U+\mu_{U,V}(X-x_U),
\]
where
\begin{equation}\label{eq:slope}
 \mu_{U,V}=
 \begin{cases}
 \dfrac{y_V-y_U}{x_V-x_U},&U\ne V,\\[7pt]
 -\dfrac{\Fpol_X(U)}{\Fpol_Y(U)}
 =\dfrac{3x_U^2+2a_2x_U+a_4-a_1y_U}
 {2y_U+a_1x_U+a_3},&U=V.
 \end{cases}
\end{equation}
Define $U\oplus V=W=(x_W,y_W)$ by
\begin{align}
 x_W&=\mu_{U,V}^2+a_1\mu_{U,V}-a_2-x_U-x_V,
 \label{eq:add-x}\\
 y_W&=-\Apol(x_W)-\lambda_{U,V}(x_W).
 \label{eq:add-y}
\end{align}
\end{definition}

Every denominator in \eqref{eq:slope} is nonzero when it is used.
Indeed, if $U\ne V$ and $V\ne-U$, then $x_U\ne x_V$.
If $U=V\ne-U$, then \eqref{eq:self-inverse} gives $\Fpol_Y(U)\ne0$.

\begin{lemma}[Line factorization]\label{lem:line-factor}
Let $U,V$ be finite points with $V\ne-U$, put $W=U\oplus V$, and write $\lambda=\lambda_{U,V}$.
Then
\begin{equation}\label{eq:line-factor}
 \Gpol(X)-\lambda(X)^2-\Apol(X)\lambda(X)
 =(X-x_U)(X-x_V)(X-x_W),
\end{equation}
where the first two factors coincide when $U=V$.
Thus the third intersection of $Y=\lambda(X)$ with $E$, counted with tangency, is $-W$.
In particular, $W\in E(K)$, so Definition~\ref{def:addition} is well-defined.
\end{lemma}

\begin{proof}
Set
\[
 H_\lambda(X)=\Gpol(X)-\lambda(X)^2-\Apol(X)\lambda(X)
 =-\Fpol(X,\lambda(X)).
\]
This is a monic cubic.
If $U\ne V$, then the line passes through both points, so $x_U$ and $x_V$ are roots.
If $U=V$, then $H_\lambda(x_U)=0$ and the chain rule together with \eqref{eq:slope} gives
\[
 H_\lambda'(x_U)
 =-\bigl(\Fpol_X(U)+\Fpol_Y(U)\mu_{U,U}\bigr)=0.
\]
Hence $x_U$ is a double root in the tangent case.
The coefficient of $X^2$ in $H_\lambda$ is $a_2-\mu_{U,V}^2-a_1\mu_{U,V}$.
Vieta's relation and \eqref{eq:add-x} identify the third root as $x_W$.
At that root the point on the line has ordinate $\lambda(x_W)=-\Apol(x_W)-y_W$, so it is $-W$ and lies on $E$.
The conjugation \eqref{eq:negative} preserves the equation of $E$, and hence $W\in E(K)$.
\end{proof}

For an affine linear polynomial $L(X)$, define its conjugate by
\begin{equation}\label{eq:conjugate-line}
 L^\iota(X)=-\Apol(X)-L(X).
\end{equation}
The line $Y=L^\iota(X)$ contains $-P$ whenever $Y=L(X)$ contains $P$, and direct expansion gives
\begin{equation}\label{eq:conjugate-factor}
 \Gpol-(L^\iota)^2-\Apol L^\iota
 =\Gpol-L^2-\Apol L.
\end{equation}

\begin{lemma}[No nontrivial fixed summand]\label{lem:no-fixed}
If $U\oplus V=U$, then $V=\OO$.
\end{lemma}

\begin{proof}
If $U=\OO$, the equality $U\oplus V=U$ gives $V=\OO$; if $V=\OO$, the conclusion already holds.
Suppose that $U,V$ are finite.
The case $V=-U$ would give $U\oplus V=\OO$, so $V\ne-U$ and the defining line $Y=\lambda(X)$ is nonvertical.
If $W=U\oplus V=U$, that line contains both $U$ and $-W=-U$.
Points on a nonvertical line with the same $x$-coordinate coincide, hence $U=-U$ and $\Fpol_Y(U)=0$.
By Lemma~\ref{lem:line-factor},
\[
 H_\lambda(X)=(X-x_U)^2(X-x_V),
\]
so $H_\lambda'(x_U)=0$.
On the other hand,
\[
 H_\lambda'(x_U)
 =-\bigl(\Fpol_X(U)+\Fpol_Y(U)\lambda'(x_U)\bigr)
 =-\Fpol_X(U).
\]
This contradicts nonsingularity at $U$.
\end{proof}

\begin{lemma}[Elementary cancellation identities]\label{lem:transport}
For all $U,V\in E(K)$,
\begin{align}
 -(-U)&=U,
 \label{eq:double-negative}\\
 -(U\oplus V)&=(-U)\oplus(-V),
 \label{eq:negative-sum}\\
 (U\oplus V)\oplus(-U)&=V,
 \label{eq:right-transport}\\
 U\oplus\bigl(V\oplus(-U)\bigr)&=V.
 \label{eq:left-transport}
\end{align}
Moreover, $U\oplus V=\OO$ if and only if $V=-U$.
\end{lemma}

\begin{proof}
Equation \eqref{eq:double-negative} follows from \eqref{eq:negative}.
Suppose first that $U,V$ are finite and $V\ne-U$, and put $W=U\oplus V$.
By \eqref{eq:conjugate-factor}, conjugating the chord or tangent through $U,V,-W$ gives the chord or tangent through $-U,-V,W$ with the same intersection polynomial.
It follows from Definition~\ref{def:addition} that $(-U)\oplus(-V)=-W$, which proves \eqref{eq:negative-sum} in this case.
If one point is $\OO$, or if $V=-U$, the same identity follows directly from the definition.

We next prove \eqref{eq:right-transport}.
If $U=\OO$, if $V=\OO$, or if $V=-U$, then respectively
\[
 (U\oplus V)\oplus(-U)=V\oplus\OO,
 \qquad
 U\oplus(-U),
 \qquad
 \OO\oplus(-U),
\]
and each expression equals $V$.
Assume now that $U,V$ are finite, $V\ne\OO$, and $V\ne-U$.
Lemma~\ref{lem:no-fixed} gives $W\ne U$; equivalently, $-U\ne-W$, so the sum $W\oplus(-U)$ is defined by a nonvertical chord or tangent.
The conjugate of the line through $U,V,-W$ contains $W,-U,-V$ and has the required multiplicity when $W=-U$.
It is therefore the defining line for $W\oplus(-U)$, whose third intersection point is $-V$.
Hence $W\oplus(-U)=V$.

Apply \eqref{eq:right-transport} with $-U$ in place of $U$:
\[
 ((-U)\oplus V)\oplus U=V.
\]
Commutativity, applied only to the two displayed binary sums, changes its left side into $U\oplus(V\oplus(-U))$ and proves \eqref{eq:left-transport}.
Finally, the definition assigns the value $\OO$ precisely to an inverse pair (including the cases involving $\OO$), which proves the last assertion.
\end{proof}

\begin{lemma}[Reduction to the generic configuration]\label{lem:generic-reduction}
To prove associativity, it is enough to consider finite points $P,Q,R$ satisfying
\begin{equation}\label{eq:generic-conditions}
 R\ne P,-P,
 \qquad
 Q\ne-P,-R,
 \qquad
 P\oplus Q\ne-R,
 \qquad
 Q\oplus R\ne-P.
\end{equation}
\end{lemma}

\begin{proof}
If one of the points is $\OO$, associativity follows from the identity law.
If $R=P$, then commutativity of the inner and outer binary sums gives
\[
 (P\oplus Q)\oplus P=P\oplus(P\oplus Q)
 =P\oplus(Q\oplus P).
\]
If $R=-P$, both bracketings equal $Q$ by \eqref{eq:right-transport} and \eqref{eq:left-transport}.
If $Q=-P$, then the left bracketing is $R$, while \eqref{eq:left-transport} shows that the right bracketing is also $R$.
If $Q=-R$, the analogous application of \eqref{eq:right-transport} shows that both bracketings are $P$.

It remains to compare the last two exclusions.
Put $S=P\oplus Q$.
Equation \eqref{eq:right-transport} gives $S\oplus(-Q)=P$.
Negating this equality and using \eqref{eq:negative-sum} yields
\[
 (-S)\oplus Q=-P.
\]
Thus $S=-R$ implies $Q\oplus R=-P$.
Conversely, put $V=Q\oplus R$.
Equation \eqref{eq:right-transport} gives $V\oplus(-Q)=R$; after negation and use of \eqref{eq:negative-sum},
\[
 (-V)\oplus Q=-R.
\]
If $V=-P$, this becomes $P\oplus Q=-R$, or $S=-R$.
When either equality holds, both bracketings are $\OO$.
These cases exhaust the failures of \eqref{eq:generic-conditions}.
\end{proof}

\section{The two-line elimination identity}\label{sec:identity}

Let $P,Q,R$ be finite points for which
\[
 S=P\oplus Q,
 \qquad
 T=S\oplus R
\]
are finite and are computed by nonvertical lines.
Write
\[
 Y=\ell(X)=\lambda_{P,Q}(X),
 \qquad
 Y=m(X)=\lambda_{S,R}(X).
\]
The first line contains $-S$, whereas the second contains $S$.
Consequently,
\begin{equation}\label{eq:line-values-at-S}
 \ell(x_S)=-\Apol(x_S)-y_S,
 \qquad
 m(x_S)=y_S.
\end{equation}

\begin{proposition}[Two-line elimination]\label{prop:two-line}
There are unique $c\in K$ and a unique monic quadratic $h\in K[X]$ such that
\begin{equation}\label{eq:c-and-h}
 \ell+m+\Apol=c(X-x_S),
 \qquad
 \Gpol+\ell m=(X-x_S)h.
\end{equation}
They satisfy the identity
\begin{equation}\label{eq:master-identity}
 (Y-\ell)(Y-m)
 =(X-x_S)(h-cY)+\Fpol(X,Y)
\end{equation}
in $K[X,Y]$.
Moreover,
\begin{align}
 h-c\ell&=(X-x_P)(X-x_Q),
 \label{eq:first-quadratic-factor}\\
 h-cm&=(X-x_R)(X-x_T),
 \label{eq:second-quadratic-factor}\\
 h^2+\Apol ch-c^2\Gpol
 &=(X-x_P)(X-x_Q)(X-x_R)(X-x_T).
 \label{eq:quartic-factor}
\end{align}
\end{proposition}

\begin{proof}
The first expression on the left of \eqref{eq:c-and-h} is linear and vanishes at $x_S$ by \eqref{eq:line-values-at-S}; this gives $c$.
At $x_S$, the curve equation gives
\[
 \Gpol(x_S)+\ell(x_S)m(x_S)
 =y_S^2+\Apol(x_S)y_S
 -\Apol(x_S)y_S-y_S^2=0.
\]
Thus the second expression is divisible by $X-x_S$.
Its leading term is $X^3$, so the quotient $h$ is monic of degree two.

Expanding the left side of \eqref{eq:master-identity} and subtracting $\Fpol$ gives
\[
 -\bigl(\ell+m+\Apol\bigr)Y+\ell m+\Gpol
 =(X-x_S)(h-cY),
\]
which proves the identity.
Substitution of $Y=\ell(X)$ into \eqref{eq:master-identity}, followed by Lemma~\ref{lem:line-factor}, gives
\[
 (X-x_S)\bigl(h-c\ell\bigr)
 =(X-x_P)(X-x_Q)(X-x_S).
\]
Since $K[X]$ is an integral domain, cancellation of $X-x_S$ proves \eqref{eq:first-quadratic-factor}.
Substitution of $Y=m(X)$ similarly gives
\[
 (X-x_S)\bigl(h-cm\bigr)
 =(X-x_S)(X-x_R)(X-x_T),
\]
and hence \eqref{eq:second-quadratic-factor}.
Finally,
\begin{align*}
 (h-c\ell)(h-cm)
 &=h^2-ch(\ell+m)+c^2\ell m\\
 &=h^2+\Apol ch-c^2(X-x_S)h+c^2\ell m\\
 &=h^2+\Apol ch-c^2\Gpol,
\end{align*}
where \eqref{eq:c-and-h} is used in the last two lines.
Combining this with \eqref{eq:first-quadratic-factor} and \eqref{eq:second-quadratic-factor} proves \eqref{eq:quartic-factor}.
\end{proof}

\begin{corollary}[Interpolating quadratic]\label{cor:quadratic}
If $c\ne0$ and $q=h/c$, then
\begin{equation}\label{eq:q-factorization}
 q^2+\Apol q-\Gpol
 =c^{-2}(X-x_P)(X-x_Q)(X-x_R)(X-x_T).
\end{equation}
The graph $Y=q(X)$ contains $P,Q,R$, and $-T$.
If a factor on the right side of \eqref{eq:q-factorization} is repeated at $x_0$ and $U=(x_0,q(x_0))$, then
\begin{equation}\label{eq:repeated-root-condition}
 \Fpol_Y(U)q'(x_0)+\Fpol_X(U)=0.
\end{equation}
Thus a repeated factor records tangency of the graph of $q$ to $E$ at $U$ whenever the tangent is nonvertical.
\end{corollary}

\begin{proof}
Divide \eqref{eq:quartic-factor} by $c^2$.
Equations \eqref{eq:first-quadratic-factor} and \eqref{eq:second-quadratic-factor} show that
\[
 q(x_P)=y_P,
 \quad q(x_Q)=y_Q,
 \quad q(x_R)=y_R,
 \quad q(x_T)=-\Apol(x_T)-y_T.
\]
If $x_0$ is a repeated root of the right side of \eqref{eq:q-factorization}, its derivative vanishes at $x_0$.
Differentiating the left side and using $q(x_0)=y_U$ gives
\[
 (2y_U+\Apol(x_0))q'(x_0)+a_1y_U-\Gpol'(x_0)
 =\Fpol_Y(U)q'(x_0)+\Fpol_X(U)=0,
\]
which is \eqref{eq:repeated-root-condition}.
If the tangent at $U$ is nonvertical, then $\Fpol_Y(U)\ne0$ and \eqref{eq:repeated-root-condition} is precisely the tangent-slope equation.
\end{proof}

For the right bracketing, put
\[
 V=Q\oplus R,
 \qquad
 T'=P\oplus V,
\]
and apply Proposition~\ref{prop:two-line} to the line through $Q,R$ followed by the line through $V,P$.
Denote the resulting scalar and monic quadratic by $\widetilde c$ and $\widetilde h$; whenever $\widetilde c\ne0$, put $\widetilde q=\widetilde h/\widetilde c$.

\begin{lemma}[Nonvanishing of the elimination scalars]\label{lem:c-nonzero}
Under the generic conditions \eqref{eq:generic-conditions}, both $c$ and $\widetilde c$ are nonzero.
\end{lemma}

\begin{proof}
The conditions $Q\ne-P$ and $P\oplus Q\ne-R$ make the two left-bracketing lines nonvertical, while $Q\ne-R$ and $Q\oplus R\ne-P$ do the same for the two right-bracketing lines.
Thus $c$ and $\widetilde c$ are defined by Proposition~\ref{prop:two-line}.
The conditions imply $x_P\ne x_R$.
If two of $x_P,x_Q,x_R$ coincide, the conjugation rule and \eqref{eq:generic-conditions} show that the only possibilities are
\[
 P=Q\ne R
 \qquad\text{or}\qquad
 Q=R\ne P.
\]
Thus there are exactly three patterns: all three $x$-coordinates are distinct, $P=Q$, or $Q=R$.

Suppose first that $c=0$.
Equations \eqref{eq:first-quadratic-factor} and \eqref{eq:second-quadratic-factor} give
\begin{equation}\label{eq:c-zero-polynomials}
 (X-x_P)(X-x_Q)=(X-x_R)(X-x_T).
\end{equation}
If the three $x$-coordinates are distinct, or if $P=Q$, evaluation at $X=x_R$ is a contradiction.
If $Q=R$, the first relation in \eqref{eq:c-and-h}, evaluated at $x_Q$, gives
\[
 0=\ell(x_Q)+m(x_Q)+\Apol(x_Q)
 =2y_Q+\Apol(x_Q).
\]
Hence $Q=-Q=-R$, contrary to $Q\ne-R$.
Therefore $c\ne0$.

If $\widetilde c=0$, the right-bracketing version of \eqref{eq:c-zero-polynomials} is
\[
 (X-x_Q)(X-x_R)=(X-x_P)(X-x_{T'}).
\]
Evaluation at $X=x_P$ is a contradiction when the three $x$-coordinates are distinct or when $Q=R$.
In the remaining case $P=Q$, both right-bracketing lines pass through $P$.
Their sum with $\Apol$ vanishes identically, so evaluation at $x_P$ gives $2y_P+\Apol(x_P)=0$.
Then $P=-P=Q$, contrary to $Q\ne-P$.
Thus $\widetilde c\ne0$.
\end{proof}

\section{The short Weierstrass form}\label{sec:short}

The short Weierstrass case admits the following concise specialization of the two-line identity.
Assume $\operatorname{char}K\ne2,3$, and let $E_0$ be the projective completion of the affine cubic
\begin{equation}\label{eq:short-curve}
 E_0^{\mathrm{aff}}:\qquad Y^2=G_0(X)=X^3+aX+b,
 \qquad
 4a^3+27b^2\ne0.
\end{equation}
Its point at infinity is $\OO=[0:1:0]$, and $-(x,y)=(x,-y)$ on the affine chart.

\begin{theorem}[Associativity in short Weierstrass form]\label{thm:short-associativity}
The chord--tangent operation on $E_0(K)$ is associative.
\end{theorem}

\begin{proof}
By Lemma~\ref{lem:generic-reduction}, only the conditions \eqref{eq:generic-conditions} need to be considered.
Let
\[
 S=P\oplus Q,
 \quad
 T=S\oplus R,
 \quad
 V=Q\oplus R,
 \quad
 T'=P\oplus V.
\]
For the left bracketing, let $\ell$ be the line through $P,Q,-S$ and let $m$ be the line through $S,R,-T$.
Proposition~\ref{prop:two-line}, with $\Apol=0$ and $\Gpol=G_0$, gives
\begin{equation}\label{eq:short-master}
 (Y-\ell)(Y-m)
 =(X-x_S)(h-cY)+Y^2-G_0(X),
\end{equation}
where $c$ is the sum of the slopes of $\ell$ and $m$.
By Lemma~\ref{lem:c-nonzero}, $c\ne0$, and $q=h/c$ satisfies
\begin{equation}\label{eq:short-quartic}
 q(X)^2-G_0(X)
 =c^{-2}(X-x_P)(X-x_Q)(X-x_R)(X-x_T).
\end{equation}
Thus the parabola $Y=q(X)$ passes through $P,Q,R,-T$.

Lemma~\ref{lem:c-nonzero} also gives $\widetilde c\ne0$ for the right bracketing.
It therefore yields a quadratic $\widetilde q$ through $P,Q,R,-T'$ and a factorization
\begin{equation}\label{eq:short-quartic-right}
 \widetilde q(X)^2-G_0(X)
 =\widetilde c^{-2}(X-x_P)(X-x_Q)(X-x_R)(X-x_{T'}).
\end{equation}
If $x_P,x_Q,x_R$ are distinct, ordinary quadratic interpolation gives $q=\widetilde q$.
If $P=Q$, the double factor at $x_P$ in \eqref{eq:short-quartic} gives
\[
 2y_Pq'(x_P)-G_0'(x_P)=0,
 \qquad
 q'(x_P)=\frac{G_0'(x_P)}{2y_P}.
\]
The denominator is nonzero because $P\ne-P$.
The same derivative is obtained for $\widetilde q$, and both quadratics also take the value $y_R$ at $x_R$.
Hermite interpolation therefore gives $q=\widetilde q$.
If $Q=R$, both quadratics have derivative $G_0'(x_Q)/(2y_Q)$ at $x_Q$ and take the value $y_P$ at the distinct coordinate $x_P$; Hermite interpolation gives the same conclusion.

After $q=\widetilde q$ is substituted into \eqref{eq:short-quartic} and \eqref{eq:short-quartic-right}, cancellation of the common cubic factor gives
\[
 c^{-2}(X-x_T)=\widetilde c^{-2}(X-x_{T'}).
\]
Comparison of the leading coefficients first gives $c^{-2}=\widetilde c^{-2}\ne0$, and comparison of the constant coefficients then gives $x_T=x_{T'}$.
Finally, $q(x_T)=-y_T$ and $q(x_{T'})=-y_{T'}$, so $y_T=y_{T'}$.
Hence $T=T'$, as required.
\end{proof}

\section{Associativity for the general equation}\label{sec:general-associativity}

\begin{theorem}[General Weierstrass group law]\label{thm:general-associativity}
Let $K$ be any field and let $E/K$ be the nonsingular projective curve whose affine chart is given by \eqref{eq:intro-general}.
Then the operation of Definition~\ref{def:addition} is associative.
Consequently, $E(K)$ is an abelian group with identity $\OO$ and inverse given by \eqref{eq:negative}.
\end{theorem}

\begin{proof}
The chord or tangent in Definition~\ref{def:addition} is unchanged when its two inputs are interchanged; equations \eqref{eq:add-x} and \eqref{eq:add-y} therefore show that $\oplus$ is commutative.
The identity and inverse laws are part of the definition, and Lemma~\ref{lem:line-factor} gives closure in $E(K)$.
It remains to prove associativity.
Lemma~\ref{lem:generic-reduction} disposes of all exceptional configurations using only Lemma~\ref{lem:transport}.
Assume henceforth that \eqref{eq:generic-conditions} holds and use the notation
\[
 S=P\oplus Q,
 \quad
 T=S\oplus R,
 \quad
 V=Q\oplus R,
 \quad
 T'=P\oplus V.
\]
The generic conditions ensure that all four sums are finite and their defining lines are nonvertical.
Indeed, $Q\ne-P$ makes $S$ finite, and $S\ne-R$ makes $T$ finite; these two conditions also make the corresponding defining lines nonvertical.
Likewise, $Q\ne-R$ and $V\ne-P$ make $V$ and $T'$ finite and their defining lines nonvertical.

Corollary~\ref{cor:quadratic} and Lemma~\ref{lem:c-nonzero} give quadratics $q,\widetilde q\in K[X]$ such that
\begin{align}
 q^2+\Apol q-\Gpol
 &=c^{-2}D(X)(X-x_T),
 \label{eq:left-q-factor}\\
 \widetilde q^2+\Apol\widetilde q-\Gpol
 &=\widetilde c^{-2}D(X)(X-x_{T'}),
 \label{eq:right-q-factor}
\end{align}
where
\[
 D(X)=(X-x_P)(X-x_Q)(X-x_R).
\]
Both $q$ and $\widetilde q$ take the values $y_P,y_Q,y_R$ at the corresponding $x$-coordinates.

The possible coincidence patterns were identified in the proof of Lemma~\ref{lem:c-nonzero}.
If $x_P,x_Q,x_R$ are distinct, uniqueness of quadratic interpolation gives
\begin{equation}\label{eq:q-equality}
 q=\widetilde q.
\end{equation}
Suppose next that $P=Q$.
Then $P\ne-P$ by \eqref{eq:generic-conditions}, so $\Fpol_Y(P)\ne0$ by \eqref{eq:self-inverse}.
The right side of \eqref{eq:left-q-factor} has a double zero at $x_P$.
Formal differentiation of its left side, followed by evaluation at $x_P$, gives
\begin{align*}
 0
 &=(2q(x_P)+\Apol(x_P))q'(x_P)
   +a_1q(x_P)-\Gpol'(x_P)\\
 &=\Fpol_Y(P)q'(x_P)+\Fpol_X(P).
\end{align*}
Therefore
\begin{equation}\label{eq:hermite-slope}
 q'(x_P)=-\frac{\Fpol_X(P)}{\Fpol_Y(P)}.
\end{equation}
The same calculation applied to \eqref{eq:right-q-factor} gives the same value for $\widetilde q'(x_P)$.
The two quadratics agree in value and derivative at $x_P$, and they agree in value at the distinct point $x_R$.
Their difference is divisible by $(X-x_P)^2(X-x_R)$ and has degree at most two, so it is zero.
This proves \eqref{eq:q-equality} when $P=Q$.
The case $Q=R$ is identical, with the Hermite condition imposed at $Q$.

Using \eqref{eq:q-equality} in \eqref{eq:left-q-factor} and \eqref{eq:right-q-factor}, and cancelling the nonzero polynomial $D$, gives
\begin{equation}\label{eq:linear-comparison}
 c^{-2}(X-x_T)=\widetilde c^{-2}(X-x_{T'}).
\end{equation}
The leading coefficients in \eqref{eq:linear-comparison} give $c^{-2}=\widetilde c^{-2}\ne0$, and the constant coefficients then give $x_T=x_{T'}$.
By Corollary~\ref{cor:quadratic},
\[
 q(x_T)=-\Apol(x_T)-y_T,
 \qquad
 \widetilde q(x_{T'})=-\Apol(x_{T'})-y_{T'}.
\]
Since $q=\widetilde q$ and $x_T=x_{T'}$, these equations give $y_T=y_{T'}$.
Thus
\[
 (P\oplus Q)\oplus R=T=T'=P\oplus(Q\oplus R),
\]
which proves associativity.
\end{proof}

\begin{remark}[Repeated roots]\label{rem:multiplicity}
The multiplicity information used above consists entirely of double roots of univariate polynomials.
Formal differentiation converts each such root into the first-order condition \eqref{eq:hermite-slope}, which is exactly the datum required for quadratic Hermite interpolation.
\end{remark}

\section{Characteristics two and three}\label{sec:small-characteristic}

The formulas specialize directly to characteristics $2$ and $3$.
The integers appearing in \eqref{eq:partials} and \eqref{eq:slope} are coefficients in $K$ and are reduced modulo the characteristic; neither $2$ nor $3$ occurs as a denominator.
The only tangent denominator is $\Fpol_Y(P)$, and it is used precisely when $P\ne-P$.

In characteristic $2$,
\begin{equation}\label{eq:char2-negation}
 -(x,y)=\bigl(x,y+\Apol(x)\bigr),
 \qquad
 \Fpol_Y(x,y)=\Apol(x)=a_1x+a_3.
\end{equation}
Thus a finite point is self-inverse exactly when $a_1x+a_3=0$.
At a non-self-inverse point, the tangent slope is
\begin{equation}\label{eq:char2-slope}
 \mu_{P,P}
 =\frac{x_P^2+a_4+a_1y_P}{a_1x_P+a_3}.
\end{equation}
If the denominator in \eqref{eq:char2-slope} vanishes, the point is self-inverse and Definition~\ref{def:addition} assigns $P\oplus P=\OO$; no slope is required.

In characteristic $3$,
\[
 \Gpol'(X)=-a_2X+a_4,
 \qquad
 \Fpol_Y(X,Y)=-Y+\Apol(X),
\]
and the tangent slope at a non-self-inverse point is
\begin{equation}\label{eq:char3-slope}
 \mu_{P,P}
 =\frac{-a_2x_P+a_4-a_1y_P}
 {-y_P+a_1x_P+a_3}.
\end{equation}
This is the specialization of \eqref{eq:slope} in characteristic $3$.

The interpolation argument is also independent of the size of $K$.
Over $\mathbb F_2$, the case of three distinct $x$-coordinates is vacuous, and the repeated-coordinate cases are covered by Hermite interpolation.
No scalar extension is needed.

\section{Relation to Zwegers's parabola and other proofs}\label{sec:comparison}

For the short equation \eqref{eq:short-curve}, let $\alpha$ be the slope of the line through $P,Q$ and let $\widetilde\alpha$ be the slope of the line through $P\oplus Q,R$.
The coefficient $c$ in \eqref{eq:c-and-h} is the sum of these slopes:
\begin{equation}\label{eq:c-short-slopes}
 c=\alpha+\widetilde\alpha.
\end{equation}
Zwegers denotes by $c_2$ the leading coefficient of his interpolating parabola and proves the identity \cite[Lemma~2.2]{Zwegers2024}
\[
 c_2(\alpha+\widetilde\alpha)=1.
\]
It follows that
\begin{equation}\label{eq:zwegers-identification}
 c_2=c^{-1}.
\end{equation}
His parabola is
\[
 p(X)=y_Q+(X-x_Q)(c_1+Xc_2).
\]
If $x_P,x_Q,x_R$ are distinct, both $p$ and the polynomial $q=h/c$ of Corollary~\ref{cor:quadratic} take the prescribed values $y_P,y_Q,y_R$; uniqueness of quadratic interpolation gives $p=q$.
Their leading coefficients agree as well: $h$ is monic, so the leading coefficient of $q$ is $c^{-1}=c_2$.
If $P=Q$, then $p'(x_P)=c_1+x_Pc_2=\alpha(P,P)$; if $Q=R$, the analogous identity is $p'(x_Q)=c_1+x_Qc_2=\alpha(Q,Q)$.
Thus Zwegers's defining relations make $p$ satisfy the corresponding tangent condition, while \eqref{eq:repeated-root-condition} gives the same condition for $q$.
Quadratic Hermite interpolation again yields $p=q$.

Zwegers determines the fourth intersection of this parabola with the cubic and derives coordinate formulas symmetric in the first and third input points.
The two-line identity \eqref{eq:master-identity} determines the same parabola by eliminating the intermediate pair of lines.
Associativity then follows by comparing the interpolation data arising from the two bracketings, without requiring an explicit formula for the fourth point.
Thus, in short Weierstrass form, the two constructions coincide at the level of the quadratic curve, while the subsequent algebraic comparisons are different.

The general Weierstrass equation changes the reflection in the short-form construction from $Y\mapsto-Y$ to
\[
 Y\longmapsto-\Apol(X)-Y.
\]
The corresponding quadratic substitution is no longer $q^2-\Gpol$ but
\[
 q^2+\Apol q-\Gpol.
\]
These additional terms are exactly those required to cancel the two conjugate values above the intermediate coordinate $x_S$.
Formal differentiation then supplies the repeated-point condition through $\Fpol_X$ and $\Fpol_Y$, so the same parabola argument remains valid in characteristics $2$ and $3$.

Other elementary proofs organize the associativity calculation in different ways.
Washington uses a projective nine-point configuration for two reducible cubics \cite[Section~2.4]{Washington2008}.
Friedl starts from explicit rational addition formulas and verifies the resulting polynomial identities with computer algebra \cite{Friedl2017}.
Nuida treats the general Weierstrass equation in arbitrary characteristic by finite-dimensional linear algebra, including coincident points and the field $\mathbb F_2$ \cite{Nuida2021}.
Angdinata and Xu establish the group law for general Weierstrass curves in every characteristic by a purely algebraic argument formalized in Lean \cite{AngdinataXu2023}.
In the present construction, elimination of the two successive lines reduces associativity to uniqueness of a polynomial of degree at most two subject to three ordinary or Hermite interpolation conditions; equations \eqref{eq:c-short-slopes}--\eqref{eq:zwegers-identification} identify its short-form specialization with Zwegers's parabola.


\begin{thebibliography}{99}

\bibitem{AngdinataXu2023}
D.~K. Angdinata and J.~Xu,
\emph{An elementary formal proof of the group law on Weierstrass elliptic curves in any characteristic},
in \emph{14th International Conference on Interactive Theorem Proving (ITP 2023)},
LIPIcs \textbf{268} (2023), Art.~6, 6:1--6:19.
\href{https://doi.org/10.4230/LIPIcs.ITP.2023.6}{doi:10.4230/LIPIcs.ITP.2023.6}.

\bibitem{Friedl2017}
S.~Friedl,
\emph{An elementary proof of the group law for elliptic curves},
Groups Complexity Cryptology \textbf{9} (2017), no.~2, 117--123.
\href{https://doi.org/10.1515/gcc-2017-0010}{doi:10.1515/gcc-2017-0010}.

\bibitem{Nuida2021}
K.~Nuida,
\emph{An elementary linear-algebraic proof without computer-aided arguments for the group law on elliptic curves},
International Journal of Mathematics for Industry \textbf{13} (2021), no.~1, 2150001.
\href{https://doi.org/10.1142/S2661335221500015}{doi:10.1142/S2661335221500015}.

\bibitem{Silverman2009}
J.~H. Silverman,
\emph{The Arithmetic of Elliptic Curves},
2nd ed., Graduate Texts in Mathematics, vol.~106, Springer, New York, 2009.
\href{https://doi.org/10.1007/978-0-387-09494-6}{doi:10.1007/978-0-387-09494-6}.

\bibitem{Washington2008}
L.~C. Washington,
\emph{Elliptic Curves: Number Theory and Cryptography},
2nd ed., Discrete Mathematics and Its Applications, Chapman \& Hall/CRC, Boca Raton, 2008.
\href{https://doi.org/10.1201/9781420071474}{doi:10.1201/9781420071474}.

\bibitem{Zwegers2024}
S.~Zwegers,
\emph{An elementary approach to the group law on elliptic curves},
Archiv der Mathematik \textbf{123} (2024), 477--486.
\href{https://doi.org/10.1007/s00013-024-02041-6}{doi:10.1007/s00013-024-02041-6}.

\end{thebibliography}
\end{document}